\documentclass[a4paper,11pt]{article} 

\usepackage{amsmath}
\usepackage{amscd}
\usepackage{amsfonts}
\usepackage{amsthm}
\usepackage{amssymb}
\usepackage{euscript}
\usepackage{mathrsfs}
\usepackage[all]{xy}

\newtheorem{thm}{Theorem}[subsection]
\newtheorem{lem}[thm]{Lemma}

\newtheorem{prop}[thm]{Proposition}

\theoremstyle{definition}

\newtheorem{defn-prop}[thm]{Definition-Proposition}
\newtheorem{rem}[thm]{Remark}

\newcommand{\Res}{\operatorname{Res}}

\newcommand{\R}{{\mathbb R}}

\newcommand{\Z}{{\mathbb Z}}

\newcommand{\ft}{{\mathfrak t}}

\newcommand{\cC}{{\mathcal C}}

\newcommand{\cN}{{\mathcal N}}

\newcommand{\cP}{{\mathcal P}}

\newcommand{\bbE}{{\operatorname{\bf E}}}

\newcommand{\bbP}{{\operatorname{\bf P}}}

\newcommand{\sbullet}{{\scriptscriptstyle\bullet}}

\makeatletter
\newcommand*\harasita[1]{%
        \setbox0=\hbox{$\to$}%
        \@tempcnta=0
        \@tempdimb=\ht0
        \@tempdima=\@tempdimb
        \multiply\@tempdima #1
        \advance\@tempdima -\@tempdimb
        \divide\@tempdima 2
        \mathrel{\hbox{%
                $\@whilenum\@tempcnta<#1\do{%
                        \raisebox{\@tempdima}{\copy0}
                        \advance\@tempdima -\@tempdimb
                        \kern -\wd0
                        \advance\@tempcnta 1
                }%
                \kern\wd0$%
        }}
}
\makeatother

\def\lambdabar{\protect\@lambdabar}
\def\@lambdabar{%
\relax
\bgroup
\def\@tempa{\hbox{\raise.73\ht0
\hbox to0pt{\kern.25\wd0\vrule width.5\wd0
height.1pt depth.1pt\hss}\box0}}%
\mathchoice{\setbox0\hbox{$\displaystyle\lambda$}\@tempa}%
{\setbox0\hbox{$\textstyle\lambda$}\@tempa}%
{\setbox0\hbox{$\scriptstyle\lambda$}\@tempa}%
{\setbox0\hbox{$\scriptscriptstyle\lambda$}\@tempa}%
\egroup
}

\renewcommand{\Re}{{\operatorname{Re}}}
\renewcommand{\Im}{{\operatorname{Im}}}

\usepackage[dvips]{graphicx}
\usepackage{slashbox}
\usepackage{multirow}

\newcommand{\keywords}[1]{\textbf{{Keywords:}} #1}
\newcommand{\MSC}[1]{\textbf{{2010 Mathematical Subject Classification:}} #1}

\begin{document}

\title
{\bf
%Counting Newton polygons and \\the Riemann hypothesis
%Asymptotic formula of
%the number of Newton polygons
Asymptotic formula of the number of Newton polygons
}
\author{Shushi Harashita
\thanks{Graduate School of Environment and Information Sciences, Yokohama National University.
E-mail: \texttt{harasita@ynu.ac.jp}}}
%\date{\small Dedicated to Professor ...}

\maketitle

\begin{abstract}
In this paper, 
we enumerate Newton polygons asymptotically.
%we give an asymptotic formula of the number of Newton polygons.
%(lower convex line graphs whose breaking points are lattice points).
The number of Newton polygons is computable by a simple recurrence equation,
but unexpectedly the asymptotic formula of its logarithm
contains growing oscillatory terms.
%as second main terms,
%These terms are the waves of
As the terms come from
non-trivial zeros of the Riemann zeta function,
an estimation of the amplitude of the oscillating part
is equivalent to the Riemann hypothesis.
\bigskip

\noindent\keywords{Newton polygons, Asymptotic formula, Riemann hypothesis, Central limit theorem}

\noindent\MSC{11P82, 11M06}
%This surprisingly makes the desired estimation possible.
%We also give an asymptotic formula of the number of symmetric Newton polygons,
%which are the classifying data of isogeny classes of
%$p$-divisible groups of abelian varieties.

\end{abstract}

\section{Introduction}
In many algebro-geometric contexts,
Newton polygons appear as combinatorial invariants of algebraic objects.
For instance,
a polynomial over a local field defines a Newton polygon,
which knows much about how the polynomial factors.
%they are used as a basic tool when we
%consider factorizations of polynomials
%over local fields.
It is also well-known as 
Dieudonn\'e-Manin classification (cf.\;\cite{Manin})
that isogeny classes
of $p$-divisible groups
(resp. the $p$-divisible groups of abelian varieties)
over an algebraically closed field in characteristic $p>0$
are classified by Newton polygons (resp. symmetric Newton polygons),
%The Diendonn\'e-Manin classification (cf.\;\cite{Manin})
%says that  Newton polygons (resp. symmetric Newton polygons) 
%are classifying data of isogeny classes
%of $p$-divisible groups
%(resp. the $p$-divisible groups of abelian varieties)
%over an algebraically closed field in characteristic $p>0$,
where $p$-divisible groups are also called Barsotti-Tate groups.
Also similar combinatorial data appear when we consider the Harder-Narasimhan filtration of vector bundles (cf.\;\cite{HN}).
%Finding the number of Newton polygons would be a natural question.
This paper aims to show that
Newton polygons are not only useful to study such algebraic objects
but also have an importance on the number of them.
%by enumerating Newton polygons asymptotically.
%This paper aims to enumerate Newton polygons asymptotically.

A {\it Newton polygon} of height $n$ is a lower convex line graph $\xi$
over the interval $[0,n]$ with $\xi(0)=0$ where
all breaking points of $\xi$ belong to $\Z^2$.
Our main theorem (Theorem \ref{MainTheorem1}) describes the asymptotic behavior of the number $\cN(n)$
of Newton polygons of height $n$ with slopes $\in [0,1)$ as $n\to\infty$.
%In Introduction for simplicity we assume
%that the slopes of the Newton polygons are non-negative and less than one.
%We are concerned with the asymptotic behavior of
%the number $\cN(h)$ of Newton polygons of height $h$ as $h\to\infty$.
It says that the logarithm of $\cN(n)$
oscillates around the logarithm of
\[
\cP(n):=\frac{C^{1/9}K}{\sqrt{6\pi}}\frac{1}{n^{11/18}}\exp\left(\frac{3}{2}C^{1/3}n^{2/3} \right),\]
where $C=2\zeta(3)/\zeta(2)=1.4615\cdots$
and $K=\exp(-2\zeta'(-1)-\log(2\pi)/6)=1.0248\cdots$.
The oscillating terms are the second main terms
but the coefficients of the terms
are so small that it would be hard to predict the oscillation
from any computational enumeration from the definition of $\cN(n)$,
whereas the  value of the constant $C^{1/9}K/\sqrt{6\pi}$ could be approximately predicted.
The oscillation gets larger and larger as $n$ increases.
%gets bigger and bigger, as $n$ goes to the infinity.
We shall see in Theorem \ref{MainTheorem2} that
the amplitude of the oscillating part of the logarithm of $\cN(n)$ is $O(n^{1/6+\epsilon})$ for any $\epsilon > 0$
if the Riemann hypothesis (cf.\;\cite{R} and \cite[10.1]{T}) is true and has a larger order otherwise.
%On the other hand, there is a recurrence equation which gives us a method to compute $\cN(n)$,
%from which we have a table of the exact values of $\cN(n)$ for $n\le 10000$.

This paper is organized as follows.
In Section \ref{Section_generating_function}, we find a generating
function of $\cN(n)$ and
%In Section \ref{Section_logarithm} we 
describe the logarithm of
the generating function.
Section \ref{Section_main_results}
is the main part of this paper.
Our main results are stated in Section \ref{Subsection Main results}.
In Section \ref{Section_CLT},
%is a key part of this paper
%as well as Section \ref{Section_logarithm}.
we give a proof of the asymptotic formula (Theorem \ref{MainTheorem1}), 
following the method of the paper \cite{BD} by B\'aez-Duarte,
where Hardy-Ramanujan asymptotic formula \cite{HR} for partitions of integers
was re-proved by applying
Lyapunov's central limit theorem with some tail estimations.
In Section \ref{ProofMainTheorem2} we prove the second theorem (Theorem \ref{MainTheorem2}) on
the relation between the amplitude of the oscillation and the Riemann hypothesis.
In Section \ref{Section_variants}, we treat two variants:
one is the case that slopes belong to the interval $[0,1]$
and the other is the case that Newton polygons are symmetric.
In Section \ref{Section_RecurrenceEquation}
we find a recurrence equation for the numbers of Newton polygons
and observe the asymptotic formula with numerical data.
\subsection*{Acknowledgments}
This work started with the graduation thesis (February in 2017) by Takuya Tani
(supervised by the author),
where a variant of the generating function \eqref{GF_two_variables} was found.
%I am grateful also to Shinji Saito and Hayato Senda
%the other seniors in my laboratory at that time for their discussions
%about this topic.
I would like to thank Professor Norio Konno
for helpful suggestions and supports continuing
after Tani's presentation of his graduation thesis,
which led me to the probabilistic approach.
The author also thank the anonymous referee for his/her careful reading and helpful comments.
This work was supported by
JSPS Grant-in-Aid for Scientific Research (C) 17K05196.

\section{The generating function
and its logarithm}\label{Section_generating_function}
In this section, we find a generating function for the numbers of Newton polygons, and study its logarithm.
\subsection{The generating function}
A {\it Newton polygon} of height $n$ and depth $d$
is a lower convex line graph starting at $(0,0)$ and ending at $(n,d)$
with breaking points belonging to $\Z^2$.
Let $n$ and $d$ be non-negative integers. % such that $d<h$.
Let $\rho(n,d)$ be the number of Newton polygons of height $n$ and depth $d$ with non-negative slopes $<1$.
Note
\begin{equation}\label{N is sum of rho}
\cN(n) = \sum_{d=0}^\infty \rho(n,d)
\end{equation}
with $\rho(n,d)=0$ for $d\ge \max\{1,n\}$.
%as $h\to\infty$.

A Newton polygon is expressed as a multiple set of segments,
where a segment is a pair $(k,\ell)$ of non-negative integers $k,\ell$ with $\gcd(k,\ell)=1$. %and the slope condition $0\le n/m <1$.
Indeed, for 
a multiple set $\xi:=\{(k_i,\ell_i) \mid i=1,2,\ldots,t\}$ of segments
with $n=\sum k_i$ and $d=\sum \ell_i$,
we arrange them so that $\ell_i/k_i \le \ell_j/k_j$ for $i<j$,
and to $\xi$ we associate the Newton polygon of Figure \ref{fig:NP}.
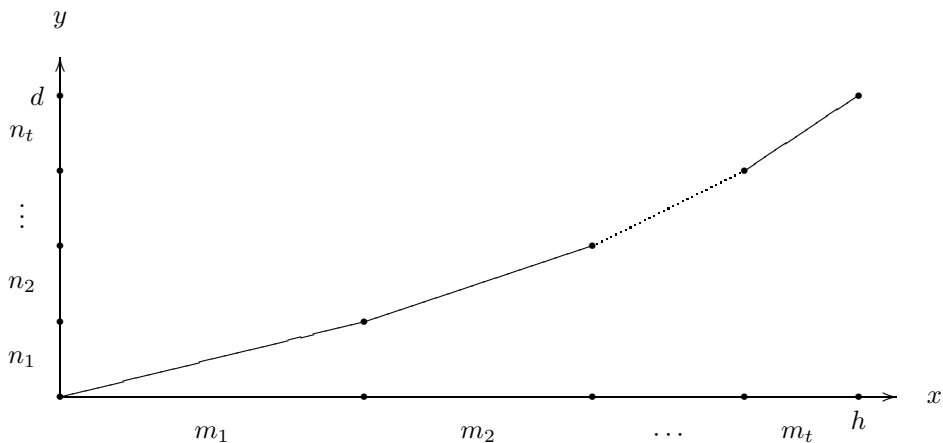
\begin{figure}[h]
\begin{equation*}
%\label{Newton polygon}
{\small
\begin{xy}
(0,0)="0"*{\sbullet},
(40,10)="1"*{\sbullet},
(70,20)="2"*{\sbullet},
(90,30)="3"*{\sbullet},
(105,40)="4"*{\sbullet},
"0";"1"**@{-},
"1";"2"**@{-},
"2";"3"**@{.},
"3";"4"**@{-},
(40,0)*{\sbullet},
(20,-5)*{k_1},
(70,0)*{\sbullet},
(55,-5)*{k_2},
(90,0)*{\sbullet},
(80,-5)*{\cdots},
(105,0)*{\sbullet},
(97,-5)*{k_t},
(105,-3)*{n},
(0,10)*{\sbullet},
(-5,5)*{\ell_1},
(0,20)*{\sbullet},
(-5,15)*{\ell_2},
(0,30)*{\sbullet},
(-5,25)*{\vdots},
(0,40)*{\sbullet},
(-5,35)*{\ell_t},
(-3,40)*{d},
%
%(0,7.5)="s"*{\sbullet},
%(55,35)="t",
%"s";"t"**@{-},
%(25,15)*{\xi},
%(15,20)*{\ell_\lambda(\xi)},
%(33,28)*{\alpha_\lambda(\xi)},
(110,0)="1100"*{},
(0,45)="040"*{},
%(-12.5,7.5)*{\displaystyle\frac{\ 1\ }{s}\langle v_\lambda,  \alpha_\lambda(\xi) \rangle},
(115,0)*{x},
(0,50)*{y},
\ar "0";"1100"
\ar "0";"040"
\end{xy}
}
\end{equation*}
 \caption{Newton polygon}
 \label{fig:NP}
\end{figure}

We have a generating function of $\rho(n,d)$ in the following form
\begin{equation}\label{GF_two_variables}
\prod_{\begin{matrix} 0\le \ell/k <1, \\ \gcd(k,\ell)=1\end{matrix}} \frac{1}{1-x^{k}y^{\ell}}=
\sum_{n=0}^\infty\sum_{d=0}^{n} \rho(n,d)x^{n}y^{d}.
\end{equation}
To see this equation, we compare the $x^ny^d$-coefficients of the both sides.
The $x^ny^d$-coefficient of the left-hand side
is the number of multiple sets $\{(k_i,\ell_i)\}$
with $0 \le \ell_i/k_i < 1$ and $\gcd(k_i,\ell_i)=1$
such that $\sum k_i = n$ and $\sum \ell_i = d$,
which is nothing other than $\rho(n,d)$.
%by comparing the $x^hy^d$-coefficient of the both sides.
%Consider
%\begin{equation}
%\phi(x,y):= \prod_{\begin{matrix} 0\le n/m <1 \\ \gcd(m,n)=1\end{matrix}} \frac{1}{1-x^{m}y^{n}} %= \sum_{h=0}^\infty\sum_{d=0}^{h-1} \rho(h,d)x^hy^d.
%\end{equation}
%Clearly one can check that this is a a generating function of $\rho(h,d)$:
%\[
%\phi(x,y) = \sum_{h=0}^\infty\sum_{d=0}^{h-1} \rho(h,d)x^{h}y^{d}.
%\]

Substituting $1$ for $y$, we get
\begin{equation}\label{the generating function of N}
f(x):=\prod_{k=1}^\infty \left(\frac{1}{1-x^k}\right)^{\varphi(k)} = \sum_{n=0}^\infty \cN(n) x^n,
\end{equation}
where $\varphi(k)$ is Euler's totient function.
%We want to know the asymptotic behavior of $\cN(h)$
%as $h\to\infty$.
The radius of convergence of $f(x)$ is one, since
\begin{equation}\label{R_of_f}
\sum_{k=1}^\infty \varphi(k)\left|\frac{x^k}{1-x^k}\right| <\frac{1}{1-|x|}\sum_{k=1}^\infty k |x|^k< \infty
\end{equation}
for any $|x|<1$.
%In the next section,
%The key step is the computation of
%\[
%f(\tau) := \prod_{n=1}^\infty (1-e^{-n\tau})^{-\varphi(n)}
%\]

\subsection{The logarithm of the generating function}\label{Section_logarithm}
We study the logarithm of $f(e^{-\tau})$ for $\Re(\tau) > 0$,
following the method of Meinardus \cite{M}.
First we expand it as
\[
\log f(e^{-\tau}) = \sum_{k=1}^\infty \varphi(k) \sum_{m=1}^\infty \frac{1}{m}  e^{-mk\tau}.
\]
By the formula of Cahen and Mellin
\[
e^{-\tau} = \frac{1}{2\pi i} \int_{\sigma_0-i\infty}^{\sigma_0+i \infty} \Gamma(s) \frac{1}{\tau^s}\ ds \quad (\Re(\tau)>0, \sigma_0 > 0),
\]
we get
\begin{eqnarray}
\log f(e^{-\tau}) 
&=& \frac{1}{2\pi i} \int_{\sigma_0-i\infty}^{\sigma_0+i \infty}
\frac{1}{\tau^s}\Gamma(s) \sum_{m=1}^\infty  \frac{1}{m^{s+1}}
\sum_{k=1}^\infty \frac{\varphi(k)}{k^s}\ ds\nonumber\\
&=&\frac{1}{2\pi i} \int_{\sigma_0-i\infty}^{\sigma_0+i \infty}
\frac{1}{\tau^s}\Gamma(s) \zeta(s+1)\frac{\zeta(s-1)}{\zeta(s)}\ ds
\label{BasicExpressionFromCahenMellin}
\end{eqnarray}
for $\sigma_0 > 2$.
A formula of $\log f(e^{-\tau})$ will be obtained by
moving the line of integration to the left.

Recall the fact \cite[Theorem 9.7]{T} that
there exists a constant $A$ such that for every $\nu$ there exists $t_\nu\in [\nu,\nu+1]$ for which
\begin{equation}\label{KeyOrderEstimation}
\left|\frac{1}{\zeta(\sigma + i t_\nu)}\right| \le t_\nu^A \quad (-1\le \sigma \le 2).
\end{equation}
It is obvious to extend the interval to an arbitrary interval
by using the well-known order-estimations as $|t|\to\infty$:
$|\zeta^{\pm 1}(\sigma + i t)| = O(1)$ for $\sigma > 1$ (cf.\;\cite[3.1]{T})
and
$|\zeta^{\pm 1}(\sigma + i t)| = O\left(|t|^{-\sigma + 1/2}\right)$ for $\sigma < 0$
(cf.\;\cite[(4.12.3)]{T}.
Also it is known that $|\zeta^{\pm 1}(1 + i t)| = O(|t|^\epsilon)$ for any $\epsilon > 0$
(cf.\;\cite[(3.5.2), (3.6.5)]{T}).
As for the estimation of $\zeta(s)$ in the critical strip, for instance
use the convexity bound $|\zeta(s)| = O\left(|t|^{(1-\sigma)/2 + \epsilon}\right)$
for any $\epsilon>0$, %for $|t|\ge 2$, 
see \cite[5.1]{T}.
From Stirling's formula (cf.\;\cite[(4.12.2)]{T}),
in an arbitrary strip $a\le \Re(s) \le b$,
for any $\varepsilon>0$ we have
\begin{equation}\label{OrderOfGamma}
\Gamma(s) = O\left(e^{-(\pi/2-\varepsilon)|\Im(s)|}\right)
\end{equation}
as $|\Im(s)|\to \infty$,
and $\Gamma(s)$ is rapidly decreasing also as $\sigma \to -\infty$.

Let $S$ be the set of poles of $\tau^{-s}\Gamma(s) \zeta(s+1)\zeta(s-1)\zeta(s)^{-1}$,
and $g_\alpha(\tau)$ its residue at $\alpha\in S$:
\begin{equation}\label{g_alpha}
g_\alpha(\tau):=\Res_{s=\alpha}\left(\tau^{-s}\Gamma(s)
\zeta(s+1)\frac{\zeta(s-1)}{\zeta(s)}\right).
\end{equation}
It is known that $\Gamma(s)$ has
no zero and has poles only at non-positive integers
and the poles are all simple.
Note that $S$ consists of $2,0$ and
the (non-trivial and trivial=negative even) zeros of $\zeta(s)$.
Let $c_0$ be a real number with $c_0 \ge 1$ or $c_0<0$, and
assume that $c_0$ is not equal
to the real part of any element $\alpha$ of $S$.
Then $\log f(e^{-\tau})$ is equal to
\begin{equation}\label{sum_and_integral}
\sum_{\alpha\in S \text{ s.t. } \Re(\alpha) > c_0 } g_\alpha(\tau) + \frac{1}{2\pi i} \int_{c_0-i\infty}^{c_0+i \infty}
\frac{1}{\tau^s}\Gamma(s) \zeta(s+1)\frac{\zeta(s-1)}{\zeta(s)}\ ds,
\end{equation}
where for $c_0<0$ the sum is precisely
\begin{equation}
\lim_{\nu\to\infty}
\sum_{\begin{matrix}\alpha\in S \text{ s.t. } \Re(\alpha) > c_0, \\ |\Im(\alpha)| < t_\nu\end{matrix}}
g_\alpha(\tau)
\end{equation}
with $t_\nu$ as in \eqref{KeyOrderEstimation} and this converges.
As the integral in \eqref{sum_and_integral} is $O(|\tau|^{-c_0})$
by \eqref{OrderOfGamma}
with the other order-estimations reviewed above, we obtain
\begin{equation}\label{ClearOrderEstimation}
\log f(e^{-\tau})
= \sum_{\alpha\in S \text{ s.t. } \Re(\alpha) > c_0 } g_\alpha(\tau) + O(|\tau|^{-c_0}).
\end{equation}
%
%It is easy to extend the interval to $-2 \le \sigma \le 3$
%because $|\zeta^{-1}(\sigma + i t)| \le \zeta(\sigma)$ for $\sigma > 1$
%and by the functional equation for $\Re(s) \ge 1$
%\[
%\left|\frac{1}{\zeta(1-s)}\right| = 
%\left|\pi^{s-1/2}\frac{\Gamma((1-s)/2)}{\Gamma(s/2)}\frac{1}{\zeta(s)}\right|
%=O\left(\left|\frac{\Gamma((1-s)/2)}{\Gamma(s/2)}\right|\right)
%\]

%\[
%\frac{\zeta'}{\zeta}(s) = \sum_{|t-\gamma|<1}\frac{1}{s-\rho} + O(log(|t|+2))
%\quad (-1\le \Re(s) \le 2, |s-1|\ge \delta).
%\]
%
%Note that 
%$\Gamma(s) \zeta(s+1)\zeta(s-1)\zeta(s)^{-1}$
%has (simple) poles at $s=2, 0$ and
%the (non-trivial and trivial=negative even) zeros of $\zeta(s)$.

%where $\Re(s) > -2$.
%since
%$E^*(z,s/2) = \pi^{-s/2}\Gamma(s/2)\zeta(s)E(z,s/2)$
%is regular except for simple poles at $s=0$ and $s=2$.

The residue of each pole is as follows.
First the residue at $s=2$ is
\begin{equation}\label{ResidueAtTwo}
g_2(\tau) = \frac{\zeta(3)}{\zeta(2)}\cdot\frac{1}{\tau^2}.
\end{equation}

Let $\gamma$ be a non-trivial zero of $\zeta(s)$.
If $\gamma$ is a simple zero, then
\begin{equation}
g_\gamma(\tau) = 
\frac{\Gamma(\gamma)\zeta(\gamma+1)\zeta(\gamma-1)}{\zeta'(\gamma)}\cdot \frac{1}{\tau^{\gamma}}.
\end{equation}
In general, $g_\gamma(\tau)$ is of the form
\begin{equation*}
g_\gamma(\tau) = \tau^{-\gamma}P_\gamma(\log \tau)
\end{equation*}
for some polynomial $P_\gamma$ whose degree is the order at $\gamma$ of $\zeta(s)$ minus $1$.
%Specifically
%$g_\gamma(\tau) = O(|\tau|^{-(\gamma+\epsilon)})$ as $\tau\to 0$
%for any $\epsilon > 0$.
%Considering the integral over a rectangle surrounding a single pole $\gamma$ 
%we also have
%\[
%g_\gamma(\tau) = O(e^{-(\pi/2)\Im\gamma+\epsilon})
%\]
%as $\Im\gamma\to \infty$.

The order of the pole at $s=0$ of $\Gamma(s)\zeta(s+1)\zeta(s-1)\zeta(s)^{-1}$ 
is two. We have $g_0(\tau) = D'(0) - D(0) \log \tau$
with $D(s)=\zeta(s-1)/\zeta(s)$, whence
\begin{equation}\label{ResidueAtZero}
g_0(\tau) = - \frac{1}{6}\log \tau - 2\zeta'(-1)-\frac{1}{6}\log(2\pi).
\end{equation}

We forgo determining $g_\alpha(\tau)$ for $\Re(\alpha) \le -2$,
as those contributions to our asymptotic formula are
much smaller than the errors appearing in Section \ref{Section_CLT}.
%with $\zeta'(-1) = 1/12 - \log A = -0.1654211437$,
%where $A$ is the Glaisher-Kinkelin constant.
\begin{prop}\label{Prop log of f}
For any $\epsilon>0$, we have
%\begin{equation}\label{ExplictFormLogf}
\begin{eqnarray*}
\log f(e^{-\tau}) &=& \frac{\zeta(3)}{\zeta(2)}\tau^{-2} + \sum_{\gamma}g_\gamma(\tau) 
- \frac{1}{6}\log \tau - 2\zeta'(-1)-\frac{1}{6}\log(2\pi)  +O(|\tau|^{2-\epsilon})\\
&=&\frac{\zeta(3)}{\zeta(2)}\tau^{-2} + O(\tau^{-1}).
\end{eqnarray*}
%\end{equation}
as $\tau \to 0$,
where $\gamma$ runs through non-trivial zeros of $\zeta(s)$.
%Moreover
%\begin{equation}\label{OrderNonTrivialZero}
%\[
%\sum_{\gamma}g_\gamma(\tau) = O(\tau^{-1}).
%\]
%\end{equation}
%Note that $\displaystyle g_\gamma(\tau)
%=\Gamma(\gamma)\zeta(\gamma+1)
%\zeta(\gamma-1)\zeta'(\gamma)^{-1}\tau^{-\gamma}$ if $\gamma$ is simple. 
% = \gamma \Gamma(\gamma)\zeta(\gamma+1)\zeta(\gamma-1)/\zeta'(\gamma)$ and
%$C_0=$.
\end{prop}
\begin{proof}
The first equality
follows from \eqref{ClearOrderEstimation} for $c_0=-2+\epsilon$ with \eqref{ResidueAtTwo} and \eqref{ResidueAtZero}. The second one follows from \eqref{ClearOrderEstimation} for $c_0=1$.
\end{proof}

\begin{rem}\label{difflogf}
In the same way, for any natural number $\nu$, one can show
\begin{eqnarray*}
\left(-\frac{d}{d\tau}\right)^\nu \log(e^{-\tau})
&=& (\nu+1)! \frac{\zeta(3)}{\zeta(2)}\frac{1}{\tau^{\nu+2}}
+ (-1)^\nu\sum_\gamma g_\gamma^{(\nu)}(\tau)\\
&& +\frac{(\nu-1)!}{6}\frac{1}{\tau^\nu} + O(|\tau|^{2-\nu-\epsilon})\\
&=&(\nu+1)! \frac{\zeta(3)}{\zeta(2)}\frac{1}{\tau^{\nu+2}} + O(\tau^{-1-\nu})
\end{eqnarray*}
for any $\epsilon > 0$, from the $\nu$-th derivative of \eqref{BasicExpressionFromCahenMellin}
\begin{equation*}
\left(-\frac{d}{d\tau}\right)^\nu\log f(e^{-\tau}) 
=\frac{1}{2\pi i} \int_{\sigma_0-i\infty}^{\sigma_0+i \infty}
\frac{1}{\tau^{s+\nu}}\Gamma(s+\nu) \zeta(s+1)\frac{\zeta(s-1)}{\zeta(s)}\ ds
\end{equation*}
for $\sigma_0 > 2$.
\end{rem}

\section{The asymptotic formula}\label{Section_main_results}
We state the main results on the asymptotic formula of $\cN(n)$
and on relations to the Riemann hypothesis
in Section \ref{Subsection Main results},
and prove them in later subsections.
\subsection{Main results}\label{Subsection Main results}
Here is the main result on the asymptotic formula of $\cN(n)$.
\begin{thm}\label{MainTheorem1}
We have
\[
\cN(n)
\sim \frac{C^{1/9}K}{\sqrt{6\pi}}\frac{1}{n^{11/18}}
\exp\left(\frac{3}{2}C^{1/3}n^{2/3}+\sum_{\gamma} g_\gamma\left(C^{1/3}n^{-1/3}\right)\right)
\]
with $\displaystyle C=2\frac{\zeta(3)}{\zeta(2)}$ and
$\displaystyle K=\exp\left(-2\zeta'(-1)-\frac{1}{6}\log(2\pi)\right)$,
where $\gamma$ runs through non-trivial zeros of $\zeta(s)$,
and the sum is precisely defined to be
\begin{equation*}
\lim_{\nu\to\infty}\sum_{|{\operatorname{Im}}(\gamma)| < t_\nu} g_\gamma(\tau),
\end{equation*}
with the notation of \eqref{KeyOrderEstimation},
see \eqref{g_alpha} for $g_\gamma(\tau)$.
Note that
\[
g_\gamma(\tau) = \frac{\Gamma(\gamma)\zeta(\gamma+1)\zeta(\gamma-1)}{\zeta'(\gamma)} \frac{1}{\tau^{\gamma}}
\]
%with
%\[
%c_\gamma = \frac{\Gamma(\gamma)\zeta(\gamma+1)\zeta(\gamma-1)}{\zeta'(\gamma)}
%\]
for simple zeros $\gamma$.
%and
%\[\displaystyle G_\ns(n) = \sum_{\gamma:\ \text{\rm non-simple}} g_\gamma(C^{1/3}n^{-1/3}),\]
%Note that $G_\ns(n)$ is expected to be zero, since
%one expect that there is no non-simple zero.
\end{thm}

As in Introduction, we put
\begin{equation}
\cP(n):=\frac{C^{1/9}K}{\sqrt{6\pi}}\frac{1}{n^{11/18}}\exp\left(\frac{3}{2}C^{1/3}n^{2/3} \right).
\end{equation}
By Theorem \ref{MainTheorem1}
above $\log \cN(n)$ is equivalent to $\log \cP(n)$
as $n \rightarrow+\infty$; indeed the difference is
%where its amplitude is a priori $O(n^{1/3})$ by
\[
\sum_{\gamma} g_\gamma\left(C^{1/3}n^{-1/3}\right) = O(n^{1/3})
\]
(cf.\;Proposition \ref{Prop log of f}).
In Proposition \ref{FinerEstimationOfOcillation},
this bound will be refined to
\[
O\left(n^{1/3-c(\log\log n)^{-2/3}(\log\log\log n)^{-1/3}}\right)
\]
for some constant $c>0$.
%We shall show in Section \ref{ProofMainTheorem2} that
A sharp estimation of the difference is equivalent to the Riemann hypothesis:
\begin{thm}\label{MainTheorem2}
\begin{enumerate}
\item[\rm (1)] The Riemann hypothesis is true if and only if
\begin{equation*}
|\log \cN(n) - \log \cP(n)| = O\left(n^{1/6+\epsilon}\right) 
\end{equation*}
for any $\epsilon > 0$.
\item[\rm (2)] The Riemann hypothesis is true
and all the non-trivial zeros of $\zeta(s)$ are simple if
\begin{equation*}\label{ConjectualStrongEstimation}
|\log \cN(n) - \log \cP(n)| = O\left(n^{1/6}\right) 
\end{equation*}
holds.
\end{enumerate}
%and every zero of $\zeta(s)$ is simple.
%(This is also equivalent to that
%every non-trivial zeros $\gamma$ of $\zeta(s)$ at which the residue $c_\gamma$ %of $\log f(e^{-\tau})$ is not zero
% are on the line of $\Re(s)=1/2$.)
\end{thm}

%It would be natural to expect that \eqref{ConjectualStrongEstimation} holds
%without any hypothesis.

\subsection{Proof of Theorem \ref{MainTheorem1}}\label{Section_CLT}
This paper follows the method by B\'aez-Duarte \cite{BD},
where he applied 
a probabilistic  approach to re-proving
Hardy-Ramanujan's asymptotic formula \cite{HR}
for partitions of integers.
%works for our situation.
%a Newton polygon is a sort of partition of vector.

In general, let
\begin{equation}
f(t) = \sum_{n=0}^\infty a_n t^n
\end{equation}
be a power series with $a_n \ge 0$ and positive radius $R$ of convergence.
To each $t$ with $0<t<R$, an integral random variable $X_t$ is associated so that
\begin{equation}
\bbP(X_t = n) = \frac{a_n t^n}{f(t)}.
\end{equation}
The characteristic function is given by
\begin{equation}
\bbE(e^{i\theta X_t}) = \frac{f(te^{i\theta})}{f(t)}.
\end{equation}
The mean and the variance are
\begin{equation}\label{mean and variance}
m(t) = t \frac{d}{dt} \log f(t),\quad \sigma^2(t) = t \frac{d}{dt} m(t) \quad
\end{equation}
respectively.
More generally $\left(t\cdot d/dt\right)^\nu \log f(t)$
%\begin{equation}
%\left(t\frac{d}{dt}\right)^\nu \log f(t)
%\end{equation}
is equal to the $\nu$-th cumulant (also called semi-invariant).
The $\nu$-th moment $\alpha_\nu(t) := \bbE((X_t-m(t))^\nu)$
%\begin{equation}
%\alpha_\nu(t) := \bbE((X_t-m(t))^\nu) %= \sum (n-m(t))^\nu \frac{a_nt^n}{f(t)}
%\end{equation}
is described by a polynomial in the cumulants, especially we have
\begin{eqnarray}\label{alpha4}
%\alpha_3(t) &=& \left(t\frac{d}{dt}\right) \sigma^2(t),\\
\alpha_4(t) &=& \left(t\frac{d}{dt}\right)^4 \log f(t) + 3 \sigma^4(t),
\end{eqnarray}
see \cite[Chap. IV, 2]{C}.

Let us return to our situation
\begin{equation}\label{our situation}
f(t) = \sum_{n=0}^\infty \cN(n) t^n.
\end{equation}
As seen in \eqref{R_of_f}, the radius of convergence of $f(t)$ is $1$.
From now on $X_t$, $m(t)$ and $\sigma^2(t)$ stand for the random variable,
the mean and the variance for \eqref{our situation} respectively.
By \eqref{mean and variance} and Remark \ref{difflogf}, we have
%\begin{eqnarray}
\begin{equation}\label{MeanOrder}
m(e^{-\tau}) = - \frac{d}{d\tau}\log f(e^{-\tau})
%&=& 2\frac{\zeta(3)}{\zeta(2)} \tau^{-3} - \sum_{\gamma} g'_\gamma(\tau) 
%+ \frac{1}{6} \tau^{-1} + O(|\tau|^{1-\epsilon})\nonumber\\
=  2\frac{\zeta(3)}{\zeta(2)} \tau^{-3} + O(\tau^{-2})
\end{equation}
%\end{eqnarray}
and
\begin{eqnarray}\label{SigmaSquare}
\sigma^2(e^{-\tau}) = - \frac{d}{d\tau} m(e^{-\tau})
%&=& 6\frac{\zeta(3)}{\zeta(2)}\tau^{-4} + \sum_{\gamma} g''_\gamma(\tau)
%+ \frac{1}{6}\tau^{-2} + O(|\tau|^{-\epsilon}) \nonumber\\
&=&6\frac{\zeta(3)}{\zeta(2)}\tau^{-4} + O(\tau^{-3}).
\end{eqnarray}
It follows from \eqref{SigmaSquare} that
\begin{equation}\label{Sigma}
\sigma(e^{-\tau}) =\sqrt{6\frac{\zeta(3)}{\zeta(2)}}\tau^{-2} + O(\tau^{-1}).
\end{equation}

Recall the definition of $f(t)$ given in \eqref{the generating function of N}:
\begin{equation}\label{product of f_k}
f(t)=\prod_{k=1}^\infty f_k(t)^{\varphi(k)}
\end{equation}
with
\[
f_k(t) = (1-t^k)^{-1} = 1 + t^k + t^{2k} + \cdots.
\]
We may also consider the random variable $X_{t,k}$ for $f_k(t)$.
Let $m_k(t)$ and $\sigma^2_k(t)$ be 
the mean and the variance for $X_{t,k}$, respectively.
The product \eqref{product of f_k} means that
the random variables $\{\varphi(k)\text{-copies of }X_{t,k}\}_k$ are stochastically independent.
%Consider the normalized random variable
%\[
%Z(t) := \frac{X_t - m(t)}{\sigma(t)}.
%\]
%Recall that the central limit theorem is proved by
%the convergence
%\[
%\bbE(e^{i\theta Z(t)}) \to e^{-(1/2)\theta^2}.
%\]

With respect to the normalized random variable
\begin{equation}
Z(t) = \frac{X_t-m(t)}{\sigma(t)},
\end{equation}
the coefficient $\cN(n)$ of $f(t)$ is described as
\begin{eqnarray}
\cN(n) &=& \frac{1}{2\pi t^n}\int_{-\pi}^\pi f(te^{i\theta}) e^{-in\theta} d\theta\nonumber\\
&=& \frac{f(t)}{2\pi\sigma(t)t^n}\int_{-\pi\sigma(t)}^{\pi\sigma(t)} \bbE(e^{i\vartheta Z(t)})e^{-i\vartheta\frac{n-m(t)}{\sigma(t)}}\ d\vartheta
\end{eqnarray}
with $\vartheta = \sigma(t)\theta$.
Now we choose $\tau_n$ so that
\begin{equation}
m(e^{-\tau_n}) = n
\end{equation}
and put $t_n = e^{-\tau_n}$. Then
\begin{equation}\label{Description of a_n}
\cN(n) = \frac{f(t_n)}{2\pi\sigma(t_n)t_n^n }
\int_{-\pi\sigma(t_n)}^{\pi\sigma(t_n)} \bbE(e^{i\vartheta Z(t_n)}) d\vartheta.
\end{equation}
%Since
%\[
%t_n^{-n} = \exp({n\tau_n}) = \exp({m(e^{-\tau_n})\tau_n})
%\sim \exp\left(C\tau_n^{-2}+\sum_{\gamma}\gamma c_\gamma \tau_n^{-\gamma} + 1/6\right).
%\]
From the next proposition, we obtain
\begin{equation}
\cN(n) \sim \frac{f(t_n)}{\sqrt{2\pi}\sigma(t_n)t_n^n },
\end{equation}
since obviously $\sigma(t_n)\to\infty$ as $n\to\infty$.

\begin{prop}
$f(t)$ satisfies the strong Gaussian condition, % (\cite[p.~116]{BD}),
i.e.,
\[
\int_{-\pi\sigma(t)}^{\pi\sigma(t)}
\left|\bbE(e^{i\vartheta Z(t)}) - e^{-\vartheta^2/2}\right| d \vartheta \to 0
\]
as $t \to 1$.
\end{prop}
\begin{proof}
Put $Y_{t,k} = X_{t,k} - m_k(t)$.
From the equation \eqref{alpha4}, we have
\begin{equation*}
\bbE(Y_{t,k}^4) = \left(t \frac{d}{dt}\right)^4 \log\frac{1}{1-t^k} + 3\sigma_k(t)^4
= \frac{k^4(t^{3k}+7t^{2k}+t^k)}{(1-t^k)^4}.
\end{equation*}
%since
%\[
%\left(t \frac{d}{dt}\right)^4 \log\frac{1}{1-t^k} = 
%\frac{k^4 (t^{3k}+4t^{2k}+t^k)}{(1-t^k)^4}\quad \text{and}\quad 
%\sigma_k(t)^4 = \frac{k^4 t^{2k}}{(1-t^k)^4}.
%\]
Put
\[
F_k(t)=\left(t \frac{d}{dt}\right)^3 \log\frac{1}{1-t^k}=\frac{k^3(t^{2k}+t^k)}{(1-t^k)^3}.
\]
There is a constant $c_0>0$ such that
\begin{equation}\label{4to3}
\bbE(Y_{t,k}^4)^{3/4} \le c_0 F_k(t^{3/4})
%\frac{3\sqrt{3}}{2}\left(t \frac{d}{dt}\right)^3 
%\log\frac{1}{1-t^k} + \frac{k^{3}t^{3k/4}}{(1-t^k)^{5/2}}
%\bbE(Y_{t,k}^4) \le 
%\frac{3}{2}\left(t \frac{d}{dt}\right)^4 \log\frac{1}{1-t^k}
\end{equation}
for $0\le t < 1$,
%with $c_0=81\sqrt{3}/128$.
since $(1-t^k)^{-3} \le (1-t^{3k/4})^{-3}$ and 
$(t^{2k}+7t^k+1)^3/(t^{3k/4}+1)^4$ is bounded.
%obviously there is a constant $c$
%such that $(x^2+7x+1)^3 \le c(x^{3/4}+1)^4$ for $x\in [0,1]$.
%which is obvious, as $(x^2+7x+1)^3$ is bounded and
%$(x^{3/4}+1)^4$ is positive over $[0,1]$.
%
%for $t^k> 0.05$, where
%\begin{equation*}
%\left(t \frac{d}{dt}\right)^3 \log\frac{1}{1-t^k} = 
%\end{equation*}
%Hence
%\begin{equation}
%\sum_{k=1}^\infty \varphi(k) \bbE(Y_{t,k}^4) \le 
%\frac{3}{2} \left(t \frac{d}{dt}\right)^4 \log f(t) \sim 
%180 \frac{\zeta(3)}{\zeta(2)}\frac{1}{\tau^6}=O(\tau^{-6})
%\end{equation}
%Fix $t$ for a while and put $N=[0.05/\log(t)]$.
By a standard inequality (cf.\;\cite[Chap.~III, (20)]{C}), we have
%By a well-known inequality \cite[Chap.~III, (20)]{C}
\begin{equation*}
\sum_{k=1}^\infty \varphi(k) \bbE(|Y_{t,k}|^3) 
\le 
\sum_{k=1}^\infty \varphi(k) \bbE(|Y_{t,k}|^4)^{3/4} \le c_0 F(t^{3/4}),
\end{equation*}
where
\[
F(t)=\sum_{k=1}^\infty \varphi(k) F_k(t)
%\left(t \frac{d}{dt}\right)^3 \log\frac{1}{1-t^k}
=\left(-\frac{d}{d\tau}\right)^3 \log f(e^{-\tau}) =  O(\tau^{-5})
\]
with $t=e^{-\tau}$.
%since
%$(-d/d\tau)^3 \log f(e^{-\tau}) = c_0 \tau^{-5}$.
%\begin{equation}
%\left(-\frac{d}{d\tau}\right)^3 \log f(e^{-\tau}) = O(\tau^{-5}).
%\end{equation}
Thus we have
%Using $\bbE(Y_{t,k}^4) \le 9 k^4t^k/(1-t^k)^4$ (cf.\;\cite[(2.5)]{BD}),
%we have
%\begin{eqnarray*}
%\sum_{k=1}^\infty \varphi(k) \bbE(|Y_{t,k}|^3)
%&\le& \sum_{k=1}^\infty k \bbE(Y_{t,k}^4)^{3/4}
%\le 9^{3/4}\sum_{k=1}^\infty  \frac{k^4e^{-(3/4)\tau k}}{(1-e^{-\tau k})^3}\\
%&\sim& 9^{3/4} \int_0^\infty \frac{x^4e^{-(3/4)\tau x}}{(1-e^{-\tau x})^3}\ dx
%= \frac{c_1}{\tau^5}.
%\end{eqnarray*}
%with $\displaystyle c_1 = 9^{3/4}\int_0^\infty \xi^4e^{-(3/4)\xi}/(1-e^{-\xi})^3 d\xi$.
\begin{equation}
\sigma(t)^{-3} \sum_{k=1}^\infty \varphi(k) \bbE(|Y_{t,k}|^3) = O(\tau)
\end{equation}
as $\tau \downarrow 0$. As in particular the Lyapunov condition is satisfied,
the proof of the central limit theorem implies that
\begin{equation}
\bbE(e^{i\vartheta Z(t)}) \to e^{-\vartheta^2/2}
%\exp\left(-\frac{\vartheta^2}{2}\right).
\end{equation}
uniformly over any fixed finite interval of $\vartheta$.
By an estimate due to Lyapunov \cite[Chap.~VII, Lemma 3]{C},
in the set
$|\vartheta| \le c_1/\tau$ for some constant $c_1$, we have
\begin{equation}\label{SecondRegion}
\left|\bbE(e^{i\vartheta Z(t)})\right|
%\le \exp\left(-\frac{1}{2}\theta^2 + 
%\frac{2}{3} \frac{|\theta|^3}{\sigma^3(t)}
%\sum_{k=1}^\infty \varphi(k) \bbE(|Y_{t,k}|^3)\right)
\le e^{-\vartheta^2/3}.
%\le \exp\left(-\frac{\vartheta^2}{3}\right).
\end{equation}

For $|\vartheta| > c_1/\tau$
(in other words $|\theta|=|\vartheta/\sigma(t)|>c_1'\tau$ for some constant $c_1'$), we consider
\begin{eqnarray*}
\log \bbE (e^{i\theta X_t})
&=& \log f(te^{i\theta}) - \log f(t)\\
&=& \frac{\zeta(3)}{\zeta(2)} \left(\frac{1}{(\tau-i\theta)^2}-\frac{1}{\tau^2}\right)\\
& & + \frac{1}{2\pi i}\int_{1-i\infty}^{1+i\infty} 
\left(\frac{1}{(\tau-i\theta)^s}-\frac{1}{\tau^s}\right)\Gamma(s)\frac{\zeta(s+1)\zeta(s-1)}{\zeta(s)}\ ds.
\end{eqnarray*}
The integral is $O(\tau^{-1})$, since so is 
$(\tau-i\theta)^{-s}-\tau^{-s}$ for $\Re(s)=1$.
Hence
\begin{eqnarray*}
\log \left|\bbE (e^{i\vartheta Z(t)})\right| &=& \log \left|\bbE(e^{i\theta X_t})\right|\\
&=& -\frac{\zeta(3)}{\zeta(2)}
\frac{1+3\left(\frac{\tau}{\theta}\right)^2}{\left(1+\left(\frac{\tau}{\theta}\right)^2\right)^2}\frac{1}{\tau^2}
%\frac{\left(\frac{\theta}{\tau}\right)^2
%\left(\left(\frac{\theta}{\tau}\right)^2+3\right)}
%{\left(\left(\frac{\theta}{\tau}\right)^2+1\right)^2} 
+ O\left(\frac{1}{\tau}\right)\\
%&&+\Re \frac{1}{2\pi i}\int_{1-i\infty}^{1+i\infty} 
%O\left(\frac{1}{\tau}\right)\Gamma(s)\frac{\zeta(s+1)\zeta(s-1)}{\zeta(s)}\ ds\\
&\le & - c_2\frac{1}{\tau^2}
\end{eqnarray*}
for some constant $c_2 > 0$.
%Since our $|\bbE(e^{i\theta X_t})|$ is less than that in the case of 
%partitions of integers, the estimation
%from the proof of \cite[\S~2, Theorem]{BD}.
%\begin{equation}
%\log|\bbE(e^{i\theta X_t})| = - \sum_{k=1}^\infty \varphi (k)
%\sum_{\nu=1}^\infty \frac{t^k\nu}{\nu}(1-\cos(k\nu\theta))
%\end{equation}
Thus
\begin{equation}\label{ThirdRegion}
\left| \bbE(e^{i\vartheta Z(t)}) \right| \le e^{-c_2/\tau^2}
\end{equation}
holds in the set $c_1/\tau < |\vartheta| \le \pi\sigma(t)$.
Hence for arbitrary large $A$ (independent of $t$), 
for any $\epsilon > 0$
there exists $\delta>0$ such that for any $t$ with $1-\delta \le t < 1$
(i.e., $0 < \tau \le - \log(1-\delta)$),
we have
\begin{equation}\label{region1}
\int_{|\vartheta|\le A}
\left|\bbE(e^{i\vartheta Z(t)}) - e^{-\vartheta^2/2}\right| d \vartheta < \epsilon,
\end{equation}
and by \eqref{SecondRegion} the inequality
\begin{equation}\label{region2}
\int_{A\le |\vartheta| \le c_1/\tau}
\left|\bbE(e^{i\vartheta Z(t)}) - e^{-\vartheta^2/2}\right| d \vartheta 
%\le 2\int_A^\infty e^{-(1/3)\vartheta^2}\ d\vartheta 
< 6 e^{-A^2/3}
%< 6 e^{-(1/3) A^2}
\end{equation}
follows from the elementary inequality
\[
\int_A^\infty e^{-\vartheta^2/3} d\vartheta
\le \int_A^\infty \vartheta e^{-\vartheta^2/3} d\vartheta
=\left[-\frac{3}{2}e^{-\vartheta^2/3}\right]_A^\infty = \frac{3}{2}e^{-A^2/3}.
\]
Also by \eqref{ThirdRegion} we have
\begin{equation}\label{region3}
\int_{c_1/\tau\le |\vartheta|\le \pi\sigma(t)}
\left|\bbE(e^{i\vartheta Z(t)}) - e^{-\vartheta^2/2}\right| d \vartheta <
\left(\pi\sigma(t)-\frac{c_1}{\tau}\right)e^{-c_3/\tau^2}
\end{equation}
for some constant $c_3>0$.
As \eqref{region1} can be arbitrary small
if $\tau$ is sufficiently small,
\eqref{region2} can be arbitrary small
if $A$ is sufficiently large,
and \eqref{region3} goes to zero as $\tau\to 0$, we have the proposition.
\end{proof}
%\begin{rem}
%For $A>1$ we have
%\[
%\int_A^\infty e^{-a\vartheta^2} d\vartheta
%\le \int_A^\infty \vartheta e^{-a\vartheta^2} d\vartheta
%=\left[-\frac{1}{2a}e^{-a\vartheta^2}\right]_A^\infty = \frac{1}{2a}e^{-a A^2}.
%\]
%%we have
%%\[
%%\int_{-t}^t e^{-(1/2)\theta^2} d\theta = \sqrt{2\pi}+O(e^{-(1/2)t^2})
%%\]
%as $t\to\infty$.
%
%\end{rem}

It is not so easy to evaluate $f(t_n)/(\sqrt{2\pi}\sigma(t_n)t_n^n)$
even approximately
after finding $t_n$ satisfying $m(t_n)=n$.
For this, %In order to get the asymptotic formula as in the main theorem,
we make use of the asymptotic substitution lemma:
B\'aez-Duarte \cite[Lemma 1]{BD}. 
Let us recall it in our situation.
Put
\begin{eqnarray}
m_1(t) &=& C \tau^{-3},\label{m_1}\\
\sigma_1(t) &=& \sqrt{3C} \tau^{-2}\label{sigma_1} \nonumber
\end{eqnarray}
with $C=2\zeta(3)/\zeta(2)$, where $t=e^{-\tau}$.
Let $t_n'$ be the real number satisfying
\begin{equation}\label{t_n prime}
m_1(t_n') = n.
\end{equation}
In the same way to get \eqref{Description of a_n}, we have
\begin{equation*}
a_n = \frac{f(t'_n)}{2\pi\sigma_1(t_n')(t_n')^n}
\int_{-\pi\sigma_1(t_n')}^{\pi\sigma_1(t_n')}\bbE\left(e^{i\vartheta_1 Z_1(t_n')}\right) d\vartheta_1,
\end{equation*}
where $Z_1(t)=(X_t-m_1(t))/\sigma_1(t)$.
This is equal to
\begin{equation*}
\frac{f(t_n')}{2\pi\sigma_1(t_n')(t_n')^n} \frac{\sigma_1(t_n')}{\sigma(t_n')}
\int_{-\pi\sigma(t_n')}^{\pi\sigma(t_n')}\bbE\left(e^{i\vartheta Z(t_n')}\right)
e^{i\varepsilon(t_n')\vartheta} d\vartheta,
\end{equation*}
where
\begin{equation}
\varepsilon(t):=\frac{m(t) - m_1(t)}{\sigma(t)}.
\end{equation}
If $m_1(t)$ and $\sigma_1(t)$ satisfy
(i) $m_1(t)\to\infty$,
(ii) $\sigma_1(t)\sim \sigma(t)$
and
(iii) $\varepsilon(t) \to 0$ as $t\uparrow 1$,
then we have
\begin{equation}\label{AfterAsymptoticSubstitution}
a_n \sim \frac{f(t'_n)}{\sqrt{2\pi}\sigma_1(t_n')(t'_n)^n}.
\end{equation}
%where $t_n'$ is chosen so that $m_1(t_n') = n$.
By \eqref{m_1} and \eqref{t_n prime} we have
$\tau'_n = C^{1/3}n^{-1/3}$
with $t'_n = e^{-\tau'_n}$
and
\[
\sigma(t'_n) = \sqrt{3C}\left(C^{1/3}n^{-1/3}\right)^{-2}=\sqrt{3}C^{-1/6}n^{2/3}.
\]
Then it is straightforward to get Theorem \ref{MainTheorem1} from \eqref{AfterAsymptoticSubstitution}.

%As (i) is obviously satisfied and (ii) holds by \eqref{Sigma},
As (i) and (ii) follow from \eqref{MeanOrder} and \eqref{Sigma}
respectively, it remains only to show (iii).
Put
\begin{equation*}
\beta_0 := \inf\{\beta \mid \zeta(s)\ne 0 \text{ for } \Re(s) > \beta\}.
\end{equation*}
If $\beta_0 < 1$, then (iii) would also be clear.
But $\beta_0 < 1$ has not been proven so far.
However, by good fortune, we have (iii) unconditionally.
\begin{prop}
Put $u=-\log(-\log(t))$. Then
\[
\varepsilon(t) = O\left(e^{-c\cdot u (\log u)^{-2/3} (\log \log u)^{-1/3}}\right)
\]
for some constant $c>0$.
In particular, 
$\varepsilon(t)
%\displaystyle \frac{m(t) - m_1(t)}{\sigma(t)} 
\to 0$ as $t\uparrow 1$.
\end{prop}
\begin{proof}
Put
\begin{eqnarray}\label{Definition of h}
h(\tau) &:=& \log f(e^{-\tau}) - \frac{\zeta(3)}{\zeta(2)} \frac{1}{\tau^2}\nonumber\\
&=& \int_{1-i\infty}^{1+i\infty}
\frac{1}{\tau^s} \Gamma(s) \frac{\zeta(s+1)\zeta(s-1)}{\zeta(s)}\ ds.
\end{eqnarray}
As
\[
m(t) - m_1(t) = -  \frac{d}{d\tau} h(\tau)
= \int_{1-i\infty}^{1+i\infty}
\frac{1}{\tau^{s+1}} \Gamma(s+1) \frac{\zeta(s+1)\zeta(s-1)}{\zeta(s)}\ ds
\]
with $t=e^{-\tau}$, we get
\begin{equation}\label{Integral Re = 1}
\frac{m(t) - m_1(t)}{\sigma(t)} = \frac{1}{\tau^2\sigma(t)}
\int_{1-i\infty}^{1+i\infty}
\tau^{1-s} \Gamma(s+1) \frac{\zeta(s+1)\zeta(s-1)}{\zeta(s)}\ ds
%\sum_{n}\gamma_n c_{\gamma_n} \tau^{1-\gamma_n}
\end{equation}
Note that $\tau^2\sigma(t)$ converges to a non-zero constant as $t\uparrow 1$
by \eqref{Sigma}.
%The number of $\{\gamma_n \mid T\le t_n < T+1\}$ is $\sim \log(T)$
%and $c_{\gamma_n}$ is exponentially decreasing as $t_n\to\infty$.

Set
\begin{equation}\label{DefinitionPsi}
\psi(T) := (\log T)^{-2/3}(\log\log T)^{-1/3}.
\end{equation}
Recall \cite{T}, 6.19 that $\zeta(s)$ has no zero and moreover
\begin{equation}\label{ZetaInverseEstimation}
\frac{1}{\zeta(s)}= O\left(\psi(\Im(s))^{-1}\right)
\end{equation}
in the region
\begin{equation*}%\label{T (3.6.5)}
\Re(s) > 1 - A \cdot \psi(\Im(s))
%\frac{c_0}{(\log \Im(s))^{2/3}(\log\log \Im(s))^{1/3}}
\quad\text{and}\quad \Im(s) \ge T_0
\end{equation*}
for some positive constants $A$ and $T_0$,
which is known as the Vinogradov-Korobov zero-free region.
%Since in particular,
%(a weaker result than that by de la Vall\'ee Poussin),
%Thanks to the result of de la Vall\'ee Poussin
%(cf.\;\cite[Theorem 3.8]{T}) that $\zeta(s)$ has no zero in the region
%\[
%\Re s \ge 1 - \frac{c_0}{\log(\Im s)} \quad (\Im s > t_0)
%\]
%for some positive constants $c_0$ and $t_0$,
%\[
%\Re s \ge 1 - \frac{c}{\log(|\Im s|-\log(n))},
%\]
%for $c=0.034$, $\gamma \ge 705$ and $n=47.8$
Hence, the integral of \eqref{Integral Re = 1} is equal to
\begin{equation}\label{Integration over cC}
\int_{\cC}
\tau^{1-s} \Gamma(s+1) \frac{\zeta(s+1)\zeta(s-1)}{\zeta(s)}\ ds
%\sum_{n}\gamma_n c_{\gamma_n} \tau^{1-\gamma_n}.
\end{equation}
with
\begin{equation}\label{cC eq}
\cC:\quad 
\begin{cases}
\Re(s) = 
\displaystyle 1 - A \cdot \psi(\Im(s)) & \text{if } |\Im(s)| \ge T_0, \\
\cC_0 & \text{otherwise},
\end{cases}
\end{equation}
%for an arbitrary small constant $\epsilon > 0$ and
%for some positive constant $c_1$ (depending on $\epsilon$)
%chosen so that $c_1/|\Im(s)|^\epsilon \le c_0/\log^9(|\Im(s)|)$
%for $|\Im(s)|\ge t_0$,
where
%$c_1$ is a constant chosen so that
%$c_1/\Im(s) < c_0/\log^9(\Im(s))$ for $\Im(s) > t_0$ and
$\cC_0$ is a path from $1 - A \psi(T_0) - i T_0$ to
$1 - A \psi(T_0) + i T_0$
so that every zero of $\zeta(s)$ belongs to
the left side from $\cC$ and
the real part of every point on $\cC_0$ is less than one.
It is obvious that the integral of \eqref{Integration over cC} over $\cC_0$
is $O(\tau^{\delta_0})$ for some constant $\delta_0 > 0$.
%goes to zero as $\tau \to 0$.
By \eqref{ZetaInverseEstimation}
there is a constant $B>0$ such that
\[
\left|\Gamma(s+1) \frac{\zeta(s+1)\zeta(s-1)}{\zeta(s)}\right|\le e^{-B\cdot \Im(s)}
\]
over $\cC$ with $|\Im(s)|\ge T_0$.
Hence, the absolute value of \eqref{Integration over cC} over $|\Im(s)| \ge T_0$ is less than
twice of
\[
\int_{T_0}^\infty \tau^{A \psi(T)} e^{-B T} \ dT.
\]
%\[
%\int_1^\infty e^{-T}x^{\dfrac{c}{\log(T)+2}}\ dT \to 0
%\]
%as $x \downarrow 0$.
%Replacing $x^c$ by $x$ for some $a$ with some obvious estimations,
%we need to show
%\[
%\int_1^\infty e^{-T}x^{1/T}\ dT \to 0
%\]
%as $x \downarrow 0$.
Put $\tau=e^{-u}$.
Let $T_1$ and $T_2$ satisfy $A\psi'(T_1)u+B = 0$ and $2A\psi'(T_2)u+B = 0$ respectively.
We see that $T_j \to \infty$ if $u\to\infty$ for $j=1,2$. Moreover,
by
\[
\psi'(T)=-\frac{2\log\log(T)+1}{3T(\log T)^{5/3}(\log\log T)^{4/3}}
\sim - \frac{2}{3T(\log T)^{5/3}(\log\log T)^{1/3}},
\]
we get
\[
u = -\frac{B}{j A \psi'(T_j)} \sim \frac{3B}{2jA}T_j(\log T_j)^{5/3}(\log\log T_j)^{1/3}
\]
for $j=1, 2$.
%Since
%\[
%u \sim \frac{3B}{2iA}\frac{T_i\log(T_i)}{\psi(T_i)} > \frac{3B}{2iA} T_i,
%\]
Hence, for sufficient large $u$ we have
\[
T_j < u <  T_j^{1+\delta}
\]
for an arbitrary constant $\delta>0$, which implies
$ \psi(u) \sim \psi(T_j)$ for $j=1, 2$.
%Hence
%\[
%\psi(u) < \psi\left(\frac{3B}{2iA} T_i\right) \sim \psi (T_i)
%\]
%(Note that $c_3$ can be taken as it does not depend on $\epsilon$
%but $u=-(c_1/c_2)\log(\tau)$ where $c_1$ depends on $\epsilon$.)

Note that $A \psi(T) u + BT$ is minimal at $T_1$. 
%Set $t_1=(2\epsilon u)^{1/(1+\epsilon)}$ and
%$\kappa_0=\epsilon^{\frac{1}{1+\epsilon}}+\epsilon^{-\frac{\epsilon}{1+\epsilon}}$ and
%$\kappa_a=(a\epsilon)^{1/(1+\epsilon)}
%+(a\epsilon)^{-\epsilon/(1+\epsilon)}>1$ for $a=1,2$.
%Note that $\kappa_a$ is greater than $1$.
For sufficient large $u$ (so that $T_0<T_1$),
we have
\begin{eqnarray*}
&&\int_{T_0}^\infty e^{-(A \psi(T) u + BT)}\ dT =
\int_{T_0}^{T_2} e^{-(A \psi(T) u + BT)}\ dT + \int_{T_2}^\infty e^{-(A \psi(T) u + BT)}\ dT \\
&&\le  (T_2 - T_0)e^{-(A \psi(T_1) u+ BT_1)} 
+ \frac{2}{B}\int_{T_2}^\infty \left(A \psi'(T) u + B\right) e^{-(A \psi(T) u+ BT)}\ dT\\
&&\le  (T_2 - T_0)e^{-(A \psi(T_1) u+ BT_1)}  +
\frac{2}{B}e^{-(A \psi(T_2) u + BT_2)}
% \frac{2}{c_2}e^{-c_2\kappa_2 u^{1/(1+\epsilon)}} = O\left(e^{-c_3 u^{1/(1+\epsilon)}}\right)
=O(e^{- c \psi(u)u})
\end{eqnarray*}
for some constant $c>0$. % (depending on $\epsilon$)
%In particular this goes to $0$ as $u \to \infty$.
\end{proof}

As the referee suggested,
the estimation of the integral of \eqref{Integral Re = 1} can be
applied to that of the sum $\sum_\gamma g_\gamma(\tau)$
over non-trivial zeros $\gamma$ of $\zeta(s)$.

\begin{prop}\label{FinerEstimationOfOcillation}
$\sum_\gamma g_\gamma(\tau) = O(\tau^{-1+c \psi(-\log \tau)})$
as $\tau\to 0$ for some constant $c>0$, where $\psi$ is as in \eqref{DefinitionPsi}.
%where $$
%where $u=-\log(\tau)$.
\end{prop}
\begin{proof}
It follows from \eqref{sum_and_integral} for $c_0=1$ and Proposition \ref{Prop log of f} that
\[
\sum_\gamma g_\gamma(\tau) = \int_{1-i\infty}^{1+i\infty}
\tau^{-s} \Gamma(s+1) \frac{\zeta(s+1)\zeta(s-1)}{\zeta(s)}\ ds + O(\log \tau).
\]
Comparing with the integral of \eqref{Integral Re = 1}, we get
$\sum_\gamma g_\gamma(\tau) = O(\tau^{-1} e^{-c \psi(u) u})$
with $\tau = e^{-u}$.
\end{proof}

\subsection{Proof of Theorem \ref{MainTheorem2}}\label{ProofMainTheorem2}

Let us show Theorem \ref{MainTheorem2}.

\begin{proof}[Proof of Theorem \ref{MainTheorem2}.]
(1)
First we prove the ``only if"-part. % is obvious from Theorem \ref{MainTheorem}.
Assume that Riemann hypothesis is true.
Let $h(x)$ be as in \eqref{Definition of h} and set $H(x)=h(1/x)$.
For any $\delta > 0$, there exist $\varepsilon>0$ and $t_0>0$ such that 
for any $|t| > t_0$, we have $\zeta(1/2+\delta+it)^{-1} = O\left(|t|^\varepsilon\right)$,
see \cite[(14.2.6)]{T}. 
From this, for $s=1/2+\delta+it$ we have
$\Gamma(s)\zeta(s+1)\zeta(s-1)\zeta(s)^{-1} = O\left(e^{-(\pi/2-\varepsilon')|t|}\right)$ for some $\varepsilon'>0$ as $|t|\to\infty$.
Hence
\begin{eqnarray}
H(x) &=& \lim_{T\to\infty}\frac{1}{2\pi i} \int_{1-iT}^{1+iT} x^s \Gamma(s) \frac{\zeta(s+1)\zeta(s-1)}{\zeta(s)}\ ds \label{FirstExpressionOfH}\nonumber\\
&=& \lim_{T\to\infty}\frac{1}{2\pi i}
\left(\int_{1-iT}^{1/2+\delta-iT} + \int_{1/2+\delta -iT}^{1/2+\delta+iT} 
+ \int_{1/2+\delta + iT}^{1 +iT}  \right)\nonumber\\
&&\phantom{aaaaaaaaaaaaaaaaaaa}
x^s \Gamma(s) \frac{\zeta(s+1)\zeta(s-1)}{\zeta(s)}\ ds \nonumber\\
&=& O\left(\int_{-\infty}^\infty x^{1/2+\delta}e^{-(\pi/2-\varepsilon')|v|}\ dv\right) = O\left(x^{1/2+\delta}\right)
\end{eqnarray}
as $x\to \infty$.
As
\begin{equation}\label{H and N and P}
H(\sqrt[3]{n/C})\sim \log \cN(n) - \log \cP(n),
\end{equation}
we have
$|\log \cN(n) - \log \cP(n)| = O\left(n^{1/6+\epsilon}\right)$ for any $\epsilon > 0$.

We prove the ``if"-part.
Assume $|\log \cN(n) - \log \cP(n)| = O\left(n^{1/6+\epsilon}\right)$
for any $\epsilon > 0$. 
%Let $H(x)$ be as in \eqref{Definition of h}.
% := \log f(e^{-1/x}) - \zeta(3)/\zeta(2) x^2$.
By \eqref{H and N and P}
we have
$H(x) = O\left(x^{1/2+\epsilon}\right)$ as $x\to\infty$.
We use
\begin{equation}
\tilde H(x) = \log f(e^{-1/x}) - 
\begin{cases}
\displaystyle \frac{\zeta(3)}{\zeta(2)} x^2 & \text{if } x\ge 1, \\
0 & \text{otherwise}.
\end{cases}
\end{equation}
Note $\tilde H(x)=H(x)$ if $x\ge 1$.
By the Mellin transformation, we get
\begin{equation}\label{MellinTransformTildeH}
\Gamma(s)\frac{\zeta(s+1)\zeta(s-1)}{\zeta(s)} - \frac{\zeta(3)}{\zeta(2)}\frac{1}{(s-2)}
=\lim_{X\to\infty}\int_0^X \tilde H(x) x^{-s-1} dx.
\end{equation}
%\[
%s\left(\Gamma(s)\frac{\zeta(s+1)\zeta(s-1)}{\zeta(s)}
% - \frac{\zeta(3)}{\zeta(2)}\frac{1}{(s-2)}\right)
%=\int_0^\infty x^{-s} d H(x)\]
%By integration by parts, this is equal to
%\[
%\int_0^\infty d(x^{-s}H(x)) - \int_0^\infty H(x) dx^{-s}
%= \lim_{X\to\infty}X^{-s} H(X) + s \int_0^X H(x) x^{-s-1} dx.
%\]
%If $|\log \cN(n) - \log \cP(n)| = O(n^{1/6+\epsilon})$, then $|H(x)| = O(x^{1/2+\epsilon})$.
Since $\tilde H(x) = O\left(x^{1/2+\epsilon}\right)$ as $x\to\infty$ for any $\epsilon>0$
and $\tilde H(x)=O(x^{\sigma_0})$ as $x\to 0$ for any $\sigma_0>2$
by \eqref{BasicExpressionFromCahenMellin},
the right-hand side of \eqref{MellinTransformTildeH} converges over $\Re(s) > 1/2$. This implies that $\zeta(s) \ne 0$ for $\Re(s) > 1/2$.
%so does $1/\zeta(s)$.

(2)
%Assume that the Riemann hypothesis is true and 
%that all the non-trivial zeros of $\zeta(s)$ are simple.
%Then Theorem \ref{MainTheorem1} implies
%$|\log \cN(n) - \log \cP(n)|=O(n^{1/6})$,
%since for non-trivial (simple) zeros $\gamma$ of $\zeta(s)$,
%the coefficients
%$c_\gamma=\Gamma(\gamma)\zeta(\gamma+1)\zeta(\gamma-1)\zeta'(\gamma)^{-1}$
%rapidly decrease as $|\Im(\gamma)| \to \infty$.
%
%Conversely
Assume $|\log \cN(n) - \log \cP(n)|=O(n^{1/6})$.
Then $\tilde H(x)=O(x^{1/2})$ as $x\to \infty$.
In the same way as in the proof of the ``if"-part of (1),
\begin{eqnarray}
\Gamma(s)\frac{\zeta(s+1)\zeta(s-1)}{\zeta(s)} - \frac{\zeta(3)}{\zeta(2)}\frac{1}{(s-2)} &=&
\int_0^\infty \tilde H(x) x^{-s-1} dx \nonumber\\
&=& O\left(\frac{1}{\Re(s)-1/2}\right)
\end{eqnarray}
for $\Re(s) > 1/2$.
Let $\gamma$ be a non-trivial zero
and write $s=\gamma + h$. When $h\downarrow 0$, this is 
$O(1/h)$.
%$O\left(\frac{1}{h}\right)$
This would be false, if $\gamma$ were not simple.
\end{proof}

\subsection{Some variants}\label{Section_variants}
In the previous subsections,
we treated the case that all of the slopes of Newton polygons belong to $[0,1)$.
In this section, we study when they have other slope conditions.

In general, for $I \subset \R$, 
a Newton polygon of height $n$ and depth $d$
with slope-range $I$
is a lower-convex line graph in $\R^2$
starting at $(0,0)$ and ending at $(n,d)$
whose breaking points belong to $\Z^2$
and every slope is in $I$.
We denote by $\rho_I(n,d)$ the number of Newton polygons of height $n$ and depth $d$ with slope-range $I$.
%We forgo discussing a condition of $I$ to get $\rho_I(h,d)<\infty$.
Then
\begin{equation}
\prod_{\begin{matrix} \ell/k \in I, \\ \gcd(k,\ell)=1\end{matrix}} \frac{1}{1-x^{k}y^{\ell}}= \sum_{n=0}^\infty\sum_d \rho_I(n,d)x^ny^d.
\end{equation}
We are concerned with an asymptotic formula of
\[
\cN_I(n):=\sum_d \rho_I(n,d).
\]
Substituting one for $y$, we have a generating function of $\cN_I(n)$:
\begin{equation}\label{Definition of f_I and N_I}
f_I(x):=
\prod_{\begin{matrix} \ell/k \in I, \\ \gcd(k,\ell)=1\end{matrix}} \frac{1}{1-x^{k}}= \sum_{n=0}^\infty \cN_I(n)x^n.
\end{equation}
%\[
%\prod_{\begin{matrix} n/m \in I, \\ \gcd(m,n)=1\end{matrix}} \frac{1}{1-x^{m}}=% \sum_{h=0}^\infty\cN_I(h,d)x^h.
%\]
%Clearly this is the generating function of $\rho_I(h_1,h_2)$, i.e.,
%\[
%\phi_I(x,y) = \sum_{h=0}^\infty\sum_{d=0}^\infty \rho_I(h_1,h_2)x^{h_1}y^{h_2}.
%\]

\subsubsection{The case of $I=[0,1]$}
As the unique segment with slope $1$ is $(k,\ell)=(1,1)$, we have
\begin{equation}
f_{[0,1]}(x) = f(x)\cdot(1-x)^{-1}.
\end{equation}
Hence
\[
\log f_{[0,1]}(e^{-\tau}) = \log f(e^{-\tau}) -\log(1-e^{-\tau})
\]
%\begin{eqnarray*}
%&&f_{[0,1]}(e^{-\tau}) = f(e^{-\tau}) -\log(1-e^{-\tau})\\
%&&=  \frac{\zeta(3)}{\zeta(2)}\tau^{-2} +
%\sum_{\gamma}g_\gamma(\tau) -2\zeta'(-1)-\frac{1}{6}\log(2\pi) - \frac{7}{6}\log(\tau)
%+ \frac{1}{2}\tau + o(|\tau|^{2-\epsilon})
%\end{eqnarray*}
with
\[
-\log(1-e^{-\tau}) = -\log \tau + O(|\tau|).
\]
In the same way we get $\cN_{[0,1]} \sim {\tau'_n}^{-1} \cN(n)$
where $\tau_n' = C^{1/3}n^{-1/3}$. Thus
\begin{equation*}
\cN_{[0,1]}(n)
%&\sim& \frac{1}{\tau'_n} \cN(n)\\
%%&\sim&\frac{n^{1/3}}{C^{1/3}}
%\frac{C^{1/9}K}{\sqrt{6\pi}}\frac{1}{n^{11/18}}
%\exp\left(\frac{3}{2}C^{1/3}n^{2/3}+\sum_\gamma c_\gamma C^{-\gamma/3} n^{\gamma/3} + G_\ns(n)\right)\\
\sim
\frac{K}{\sqrt{6\pi}C^{2/9}}\frac{1}{n^{5/18}}
\exp\left(\frac{3}{2}C^{1/3}n^{2/3}+
\sum_\gamma g_\gamma\left(C^{1/3}n^{-1/3}\right)
%\sum_\gamma c_\gamma C^{-\gamma/3} n^{\gamma/3} + G_\ns(n)
\right)
\end{equation*}
with the same notation as in Theorem \ref{MainTheorem1}.
%\subsection{The case of $I=[0,1/2]$}

\subsubsection{Symmetric Newton polygons}
Let $p$ be a prime number, and we fix it throughout this section.
The Dieudonn\'e-Manin classification says that
the isogeny classes of $p$-divisible groups of
abelian varieties of dimension $g$
over an algebraically closed field in characteristic $p$
are classified by symmetric Newton polygons of height $2g$ and depth $g$.
It would be meaningful to give
an asymptotic formula of the number of symmetric Newton polygons.

A Newton polygon is said to be {\it symmetric} if 
the sum of its slope at $x$ and that at $2g-x$ is one
for every $x$ where the slope at $x$ is defined.
Symmetric Newton polygons are divided into the following two types.
Those of the first type are  of the form (as multiple sets of segments)
\[
\{(k_1,\ell_1), \ldots, (k_a,\ell_a), (k_a,k_a-\ell_a),\ldots,(k_1,k_1-\ell_1)\}
\]
with
\[
\ell_1/k_1 \le \cdots \le \ell_a/k_a \le 1/2
\]
and those of the second type are of the form
\[
\{(k_1,\ell_1), \ldots, (k_a,\ell_a), (2,1), (k_a,k_a-\ell_a),\ldots,(k_1,k_1-\ell_1)\}
\]
with
\[
\ell_1/k_1 \le \cdots \le \ell_a/k_a \le 1/2.
\]
The number of those with height $2g$ of the first type
is equal to $\cN_{[0,1/2]}(g)$
and that of the second type is $\cN_{[0,1/2]}(g-1)$.
Thus the number $\cN_\text{sym}(g)$ of symmetric Newton polygons of height $2g$ is
\[
\cN_\text{sym}(g) = \cN_{[0,1/2]}(g) + \cN_{[0,1/2]}(g-1) \sim 2\cdot\cN_{[0,1/2]}(g).
\]

Set $J=[0,1/2]$.
It suffices to give an asymptotic formula of $\cN_J(g)$.
The generating function for $J$ is
\begin{equation}\label{f_J}
f_J(x) = (1-x)^{-1/2}(1-x^2)^{-1/2}f(x)^{1/2}.
\end{equation}
This follows from the next two facts.
Firstly, putting $J'=[1/2,1]$, we have $f_J(x) = f_{J'}(x)$ by \eqref{Definition of f_I and N_I} the definition of $f_I$.
Secondly $f_J(x)f_{J'}(x) = f_{[0,1]}(x)(1-x^2)^{-1}$ holds, since 
%$f_J(x)f_{J'}(x)$ contains, as factors,
%is the product of factors $(1-x^k)^{-1}$ of slope $\ell/k \in [0,1] \setminus\{1/2\}$ and
the factor $(1-x^2)^{-1}$ of slope $1/2$ (resp. that of slope $\in [0,1]\setminus\{1/2\}$ )  appears twice (resp. once) in $f_J(x)f_{J'}(x)$.
%every non-negative slope $\le 1$, where only the factor of slope $1/2$ is counted doubly.

By \eqref{f_J} we get
\begin{eqnarray*}
&&\log f_J(e^{-\tau}) =
\frac{1}{2}\log f(e^{-\tau}) - \log \tau - \frac{1}{2}\log 2 + O(|\tau|)\\
&&=
\frac{\zeta(3)}{2\zeta(2)}\tau^{-2} +
\frac{1}{2}\sum_\gamma g_\gamma(\tau) -\frac{13}{12}\log \tau + \frac{1}{2}\log K - \frac{1}{2} \log 2
+ O(|\tau|).
\end{eqnarray*}
Put
\[
m_J(t) = \frac{C}{2}\tau^{-3},\quad \sigma_J(t)= \sqrt{\frac{3C}{2}}\tau^{-2}.
\]
Let $\ft_n = e^{-\tau_n''}$ be a solution of $m_J(\ft) = n$,
namely $\tau_n'' = C^{1/3}(2n)^{-1/3}$.
Likewise, $\cN_J(n)$ is asymptotically
\[
\frac{f_J(\ft_n)}{\sqrt{2\pi}\sigma_J(\ft_n)(\ft_n)^n}.
\]
By a tedious calculation, this is equal to
\begin{equation*}
\frac{K^{1/2}}{\sqrt{6\pi}C^{7/36}}\frac{1}{(2n)^{11/36}}
\exp\left(\frac{3}{4}C^{1/3}(2n)^{2/3}+\frac{1}{2}\sum_\gamma
g_\gamma\left(C^{1/3}(2n)^{-1/3}\right)
%c_\gamma C^{-\gamma/3} (2n)^{\gamma/3} + \frac{1}{2}G_\ns(2n)
\right)
\end{equation*}
with the same notation as in Theorem \ref{MainTheorem1}.

\section{A recurrence equation and numerical observation}\label{Section_RecurrenceEquation}
In this section, we give a recurrence equation of $\rho(n,d)$'s
and tables of $\rho(n,d)$ and $\cN(n)$,
and with these data we observe the asymptotic formula.
%we estimate $\log \cN(h) - \log \cP(h)$, 
%which implies that the Riemann hypothesis is true and all the zeros are simple.
%The estimation is done by an estimation of $\rho(h,d)$,
%based on a recurrence equation of $\rho(h,d)$'s.

\subsection{A recurrence equation of $\rho(n,d)$'s}
Put
\[
\phi(x,y):= \prod_{\begin{matrix}0\le \ell/k < 1,\\ \gcd(k,\ell)=1\end{matrix}} \frac{1}{1-x^ky^\ell}=\sum_{n=0}^\infty\sum_{d=0}^n \rho(n,d)x^ny^d
\]
and consider the function with a different slope-condition:
\begin{equation}
\phi^\vee(x,y):= \prod_{\begin{matrix}0< \ell/k \le 1,\\ \gcd(k,\ell)=1\end{matrix}} \frac{1}{1-x^ky^\ell}. %=\sum \rho(n,n-d)x^ny^d
\end{equation}
The following two lemmas produce a recurrence equation of $\rho(n,d)$'s.
\begin{lem}\label{LemmaDualPhiFunction}
We have
\[
\phi^\vee(x,y)=\sum_{n=0}^\infty\sum_{d=0}^n \rho(n,n-d)x^ny^d.
\]
\end{lem}
\begin{proof}
Indeed
\[
\sum_{n=0}^\infty\sum_{d=0}^n \rho(n,n-d)x^ny^d = \sum_{n=0}^\infty\sum_{d=0}^n \rho(n,n-d)(xy)^n(y^{-1})^{n-d}
\]
is equal to
\[
\phi(xy,y^{-1}) = \prod_{\begin{matrix}0\le \ell/k < 1,\\ \gcd(k,\ell)=1\end{matrix}} \frac{1}{1-(xy)^k(y^{-1})^\ell}.
\]
This is equal to
\[
\prod_{\begin{matrix}0 < \ell'/k \le 1,\\ \gcd(k,\ell')=1\end{matrix}} \frac{1}{1-x^ky^{\ell'}}=\phi^\vee(x,y),
\]
where we put $\ell'=k-\ell$.
\end{proof}
\begin{lem}\label{LemmaForRecurrence}
We have
\[
\phi(x,y) = \phi(x,xy)\phi^\vee(xy,x).
\]
\end{lem}
\begin{proof}
Set
\[
\phi_{<\frac{1}{2}}(x,y):= \prod_{\begin{matrix}0\le \ell/k < 1/2,\\ \gcd(k,\ell)=1\end{matrix}} \frac{1}{1-x^ky^\ell}
\]
and
\[
\phi_{\ge\frac{1}{2}}(x,y):= \prod_{\begin{matrix}1/2\le \ell/k < 1,\\ \gcd(k,\ell)=1\end{matrix}} \frac{1}{1-x^ky^\ell}.
\]
Clearly 
\[
\phi(x,y) = \phi_{<\frac{1}{2}}(x,y)\phi_{\ge\frac{1}{2}}(x,y).
\]
Using $k':=k-\ell$ instead of $k$ in $\phi_{<\frac{1}{2}}(x,y)$, we have
\[
\phi_{<\frac{1}{2}}(x,y)= \prod_{\begin{matrix}0\le \ell/k' < 1,\\ \gcd(k',\ell)=1\end{matrix}} \frac{1}{1-x^{k'}(xy)^\ell} = \phi(x,xy).
\]
Similarly putting $k'=\ell$ and $\ell'=k-\ell$,
\[
\phi_{\ge\frac{1}{2}}(x,y) = \prod_{\begin{matrix}0< \ell'/k' \le 1,\\ \gcd(k',\ell')=1\end{matrix}} \frac{1}{1-(xy)^{k'}x^{\ell'}} = \phi^\vee(xy,x).
\]
These show the lemma.
\end{proof}

Now we get a recurrence equation of $\rho(n,d)$'s.

\begin{prop}\label{PropForRecurrence}
We have
\[
\rho(n,d) = \sum_{\begin{matrix}n-d=\alpha+\delta,\\ d=\beta+\gamma\end{matrix}} \rho(\alpha,\beta)\rho(\gamma,\gamma-\delta),\]
where $\alpha,\beta,\gamma$ and $\delta$ run through non-negative integers.
\end{prop}
\begin{proof}
It follows from
Lemmas \ref{LemmaDualPhiFunction} and \ref{LemmaForRecurrence} that
\[
\sum_{n=0}^\infty\sum_{d=0}^n \rho(n,d)x^ny^d
\]
is equal to
\[
\left(\sum_{\alpha=0}^\infty\sum_{\beta=0}^\alpha \rho(\alpha,\beta)x^\alpha (xy)^\beta\right)\left(\sum_{\gamma=0}^\infty\sum_{\delta=0}^\gamma \rho(\gamma,\gamma-\delta)(xy)^\gamma x^\delta\right).
\]
Compare the coefficients of $x^ny^d$ of the both sides.
\end{proof}

\begin{rem}
This recurrence equation does not contain any number-theoretic operation.
%like m|n
%Euler's totient function and the m\"obius function and so on. 
But $\rho(n,d)$ should be also ``oscillating", because the maximum of $\rho(n,d)$ ($d = 0,1,\ldots, n-1$) has the same order as $\cN(n)$
in the exponential part.
We shall discuss the asymptotic formula of $\rho(n,d)$
in a separated paper.
\end{rem}

See the web page of the author \cite{HPharashita}
for $\rho(n,d)$ for $n \le 1000$
computed from this recurrence equation
and the code by Magma (\cite{Magma} and \cite{MagmaHP}).
Here is the table of $\rho(n,d)$ for small $n$.
\bigskip

%\renewcommand{\arraystretch}{1.5}
%\renewcommand{\arraystretch}{1}
%\begin{table}
%\centering{
%\caption{
%test.
%}
%\label{table:1}
\setlength\arraycolsep{1.4pt}
\setlength\tabcolsep{1.4pt}
\begin{tabular}[h]
{|c||c|c|c|c|c|c|c|c|c|c|c|c|c|c|c|}
\hline
\backslashbox{~$n$}{$d$~~} %$h\backslash d$
  & 0 & 1 & 2 & 3 & 4 & 5 & 6 & 7 & 8 & 9 & 10& 11& 12& 13& 14\\ \hline\hline
 $0$ & \phantom{0}1\phantom{0} & \phantom{0}0\phantom{0} & \phantom{0}0\phantom{0} & \phantom{0}0\phantom{0} & \phantom{0}0\phantom{0} & \phantom{0}0\phantom{0} & \phantom{0}0\phantom{0} & \phantom{0}0\phantom{0} & \phantom{0}0\phantom{0} & \phantom{0}0\phantom{0} & \phantom{0}0\phantom{0} & \phantom{0}0\phantom{0} & \phantom{0}0\phantom{0} & \phantom{0}0\phantom{0} & \phantom{0}0\phantom{0} \\ \hline
 $1$ & 1 & 0 & 0 & 0 & 0 & 0 & 0 & 0 & 0 & 0 & 0 & 0 & 0 & 0 & 0 \\ \hline
 $2$ & 1 & 1 & 0 & 0 & 0 & 0 & 0 & 0 & 0 & 0 & 0 & 0 & 0 & 0 & 0 \\ \hline
 $3$ & 1 & 2 & 1 & 0 & 0 & 0 & 0 & 0 & 0 & 0 & 0 & 0 & 0 & 0 & 0 \\ \hline
 $4$ & 1 & 3 & 2 & 1 & 0 & 0 & 0 & 0 & 0 & 0 & 0 & 0 & 0 & 0 & 0 \\ \hline
 $5$ & 1 & 4 & 4 & 3 & 1 & 0 & 0 & 0 & 0 & 0 & 0 & 0 & 0 & 0 & 0 \\ \hline
 $6$ & 1 & 5 & 6 & 5 & 3 & 1 & 0 & 0 & 0 & 0 & 0 & 0 & 0 & 0 & 0 \\ \hline
 $7$ & 1 & 6 & 9 & 9 & 7 & 4 & 1 & 0 & 0 & 0 & 0 & 0 & 0 & 0 & 0 \\ \hline
 $8$ & 1 & 7 & 12& 14& 12& 9 & 4 & 1 & 0 & 0 & 0 & 0 & 0 & 0 & 0 \\ \hline
 $9$ & 1 & 8 & 16& 20& 20& 17& 10& 5 & 1 & 0 & 0 & 0 & 0 & 0 & 0 \\ \hline
$10$& 1 & 9 & 20& 28& 31& 28& 21& 13& 5 & 1 & 0 & 0 & 0 & 0 & 0 \\ \hline
$11$& 1 & 10& 25& 38& 45& 45& 38& 27& 15& 6 & 1 & 0 & 0 & 0 & 0 \\ \hline
$12$& 1 & 11& 30& 49& 63& 68& 63& 50& 33& 17& 6 & 1 & 0 & 0 & 0 \\ \hline
$13$& 1 & 12& 36& 63& 86& 99& 98& 85& 64& 40& 20& 7 & 1 & 0 & 0 \\ \hline
$14$& 1 & 13& 42& 79&114&139&147&136&113& 80& 48& 23& 7 & 1 & 0 \\ \hline
$15$& 1 & 14& 49& 97&148&189&212&209&186&145& 98& 57& 25& 8 & 1 \\ \hline
\end{tabular}
%}
%\end{table}

\subsection{Numerical observation}
To give a table of $\cN(n)$,
we compute $\cN(n)$ by using the generating function \eqref{the generating function of N}, since it is much faster than by using the recurrence equation
(Proposition \ref{PropForRecurrence}) and \eqref{N is sum of rho}.
See the web page of the author \cite{HPharashita}
for the code by Magma (\cite{Magma} and \cite{MagmaHP})
and its log-file for the list of $\cN(n)$
for $n\le 100000$. Here is a sample:
%\bigskip

\begin{center}
\begin{tabular}[h]{|r|l|l|}\hline
$n$ & $\cN(n)$ & $\cP(n)$ \\ \hline\hline
$1$ & $1$ & $1.350821266\cdots$ \\ \hline
2 & 2 &     $2.403759900\cdots$ \\ \hline
3 & 4 &     $4.340223158\cdots$ \\ \hline
4 & 7 &     $7.696029030\cdots$ \\ \hline
5 & 13 &    $13.36116532\cdots$ \\ \hline
6 & 21 &    $22.74249494\cdots$ \\ \hline
7 & 37 &    $38.03034100\cdots$ \\ \hline
8 & 60 &    $62.59799195\cdots$ \\ \hline
9 & 98 &    $101.5922302\cdots$ \\ \hline
10 & 157 &  $162.7997475\cdots$ \\ \hline
%11 & 251 & 257.9105033\cdots \\ \hline
%12 & 392 & 404.3495439\cdots \\ \hline
%13 & 612 & 627.9191601\cdots \\ \hline
%14 & 943 & 966.5911220\cdots \\ \hline
%15 & 1439 & 1475.924378\cdots \\ \hline
%16 & 2187 &2236.770946\cdots \\ \hline
%17 & 3293 &3366.191017\cdots \\ \hline
%18 & 4930 & 5032.852302\cdots \\ \hline
%19 & 7330 & 7478.674811\cdots \\ \hline
%20 & 10839 &11049.14367\cdots \\ \hline
%21 & 15935 & 16235.61731\cdots \\ \hline
%22 & 23315 & 23734.18182\cdots \\ \hline
%23 & 33933 & 34527.26541\cdots \\ \hline
%24 & 49170 & 49996.47486\cdots \\ \hline
%25 & 70914 & 72078.14200\cdots \\ \hline
%26 & 101861 & 103477.1455\cdots \\ \hline
%27 & 145713 & 147960.0580\cdots \\ \hline
%28 & 207638 & 210755.9855\cdots \\ \hline
%29 & 294796 & 299103.3249\cdots \\ \hline
%30 & 417061 & 422993.6668\cdots \\ \hline
%50 & 241265556 & $2.435999190\cdots\ 10^8$\\ \hline
100  &$124156847482548$& $1.248747592\cdots\ 10^{14}$ \\ \hline
%500 & 205644030111987836560051739080976015928167660 & $2.060122145\cdots\ 10^{44}$ \\ \hline
1000 &
$3.061052712\cdots\ 10^{71}$
%306105271223553016602740806449653554347690712555449623916303473371601798
&$3.064377128\cdots\ 10^{71}$ \\ \hline
%5000 & $1.9884932363\cdots\ 10^{213}$ & $1.989171706\cdots 10^{213}$ \\ \hline
10000 &
$1.235480725\cdots\ 10^{340}$
%12354807259222080760538936288579508185733527885505
%60053296908256293722008276011007850321596199505521
%68434304660388678907606541788414190741633590616309
%53548010413007351358496345088452257790543547934282
%19050588852380873843641017762421194886339572849648
%76272308041705784954533299886753147460566517496364
%86934288774901088387153693245621692161529
& $1.235736815\cdots\ 10^{340}$ \\ \hline
100000 & $1.185775851\cdots\ 10^{1589}$
%26732356482134235770508046600790663595786069976485736727011802094942037554201844615182277456666929547767507188003596879952340236873987480973681052863360054222762127055431400109039453991985970342713879111749715150805176085150706160750815852576246832963082474547605289083885924452443273331952009302860594541181957373702958117700389794150772087688954990351403249985500704334281733320156783803777403505887643572816282011169737895123171981773724337437229677601057627786378282329751466070640974886831457426494735868225392316248316588401266806616748418605018010286350852635591548868321087437780125237493312366767839273122907956356487788140042256289445877382138202625360063630581231079029394918608265601519602567476406084262307618092594002305678587941884145908177576610230496905862512230277715961779970904087043713168648169670414932163852590743151068081971698251388269193855048811115997997784513727660820892961124952264542676230584191214329360277344368665735610287927069712349740127169045113037227668805598907094501048219450728838067919374328186072733942687428475241451289956378371221695191062351009495479254621694727713768059222050681632269156433699013046546078005908257936238125616023746272336781111678497075477074559849347201439046544235435061692768620352270838770677416033157478177976779101929554248623290650159005567387770576188843426017082178364434209477423036632293106775840174439907307179712294862928477643938619805352673959131615631496185596588502741731697898556543030674007393176248917174746286313083246804793181110026662335788314595506508100637662437289161237699572516522458024
& $1.185822461\cdots\ 10^{1589}$\\ \hline
\end{tabular}
\end{center}
%\bigskip

From this table,
$\cN(n)/\cP(n)$ looks to be approaching to $1$.
This is true for relatively small $n$,
but as Theorem \ref{MainTheorem1} says,
$\log \cN(n)$ oscillates around $\log \cP(n)$ for large $n$.

To see the oscillation, let us illustrate 
the contribution of the first zero $\gamma_1 = 1/2 + 14.13472514\cdots \sqrt{-1}$.
With the notation of Theorem \ref{MainTheorem1},
Figure \ref{fig:one} is the graph of
\begin{equation}\label{EquationWaveFirstZero}
y = \exp\left(2\Re\left(c_{\gamma_1} C^{-\gamma_1/3} x^{\gamma_1/3} \right)\right)
\end{equation}
with $c_\gamma = \Gamma(\gamma)\zeta(\gamma+1)\zeta(\gamma-1)\zeta'(\gamma)^{-1}$,
which is an output by Maple 2016 (\cite{MapleHP}), see the web page of the author \cite{HPharashita}.
\begin{figure}[h]
 \begin{center}
  \includegraphics[width=10cm,height=5cm]{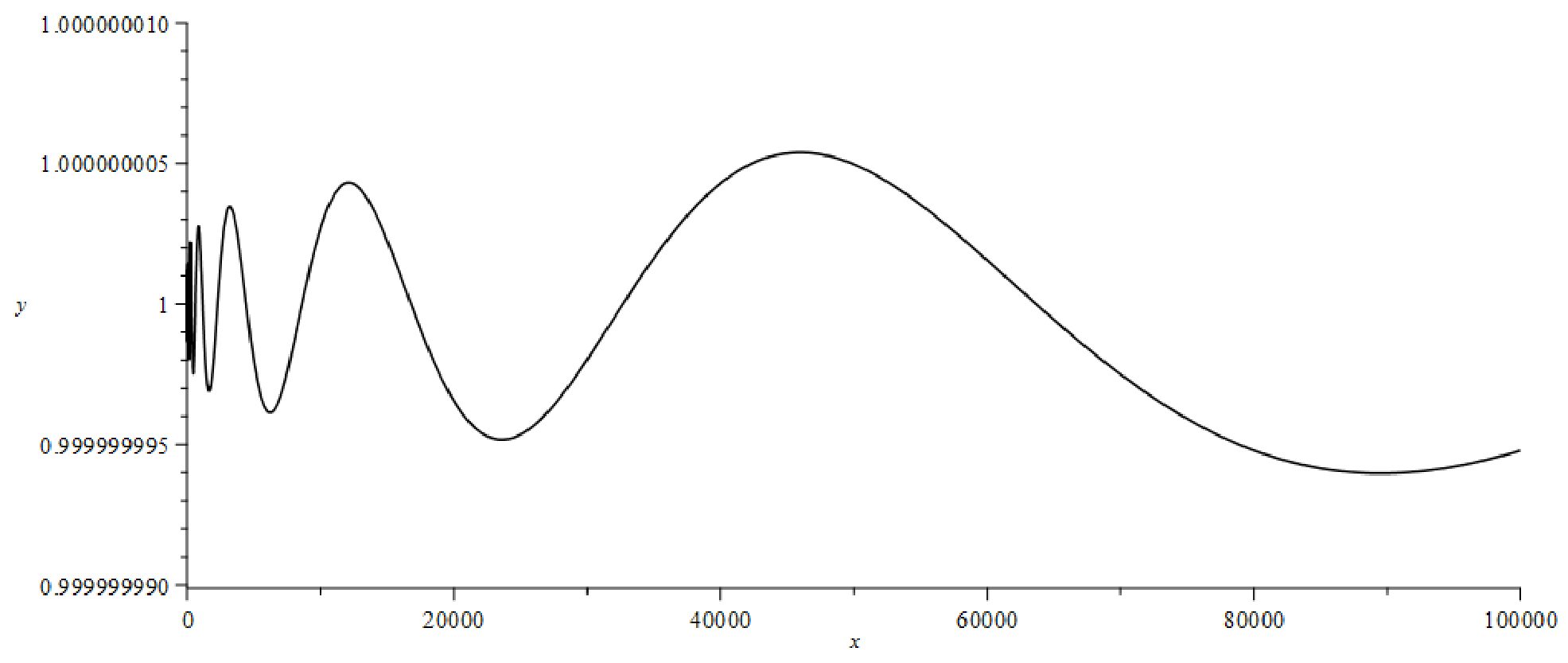}
 \end{center}
 \caption{The wave of the first zero}
 \label{fig:one}
\end{figure}
The wave of the first zero occupies most of %(and is almost equal to)
the oscillatory part,
if the Riemann hypothesis is true,
every zeros are simple and
$c_\gamma$
is of rapid decay as $\Im(\gamma) \to \infty$
(very plausible since $\Gamma(\gamma)$ is rapidly decreasing).
%If the Riemann hypothesis is true and every zeros are simple,
%$c_\gamma$ is of rapid decay as $\Im(\gamma) \to \infty$
%by \eqref{OrderOfGamma} and $|\zeta'(\gamma)^{-1}|=O(|\Im(\gamma)|)$
%(cf.~\cite{T}, 14.29.4), and therefore it is very plausible that
%the wave of the first zero occupies most of %(and is almost equal to)
%the oscillatory part.
%(very plausible since $c_\gamma$ is of rapid decay as $\Im(\gamma) \to \infty$ under the hypothesis, see \cite{T}, 14.29.4).
%$\Gamma(\gamma)$ is rapidly decreasing).
Here is a numerical data
\begin{eqnarray*}
c_{\gamma_1}&=&3.011993987\cdots 10^{-11}+4.792386731 \cdots 10^{-10}\sqrt{-1},\\
c_{\gamma_2}&=&-5.721173997\cdots 10^{-15}+1.369306521\cdots 10^{-14}\sqrt{-1},\\
c_{\gamma_3}&=&-2.705070957\cdots 10^{-17}+2.213981113\cdots 10^{-17}\sqrt{-1}
\end{eqnarray*}
for $\Im(\gamma_2) = 21.02203963\cdots$ and
$\Im(\gamma_3)=25.01085758\cdots$.
Comparing Figure \ref{fig:one} and the table of $\cN(n)$, we see that
the amplitude of the wave
is still much smaller than the error term for $x\le 100000$.
However, it gets larger as $x$ increases.
From Figure \ref{fig:one}, the reader can guess how it grows.
%This is almost equal to the whole of the oscillatory part.

\end{document}